\documentclass[reqno,11pt]{amsart}

\usepackage{amsthm}
\usepackage{amssymb}
\usepackage{amsmath}
\usepackage{amsfonts}

\usepackage[utf8]{inputenc}
\usepackage[english]{babel}

\usepackage{microtype}      \usepackage[T1]{fontenc}
\usepackage{textcomp}
\usepackage{mathptmx}

\usepackage{bookmark}

\usepackage{xcolor}
\usepackage{hyperref}

\hypersetup{pdfauthor = {Maycon Sambinelli},pdftitle = {Seven colors suffice for proper conflict-free colorings of planar graphs},pdfkeywords = {proper conflict-free coloring, odd coloring, planar graph, graph minor},pdfproducer = {Latex with hyperref},pdfcreator = {pdflatex},colorlinks,linkcolor = {red!60!black},citecolor = {green!60!black},urlcolor = {blue!60!black}}

\makeatletter
\let\origsection=\section \def\section{\@ifstar{\origsection*}{\mysection}}
\def\mysection{\@startsection{section}{1}\z@{.7\linespacing\@plus\linespacing}{.5\linespacing}{\normalfont\scshape\centering\S}}
\makeatother

\usepackage[abbrev,msc-links,backrefs]{amsrefs}
\usepackage{doi}

\renewcommand{\PrintDOI}[1]{\doi{#1}}

\usepackage{datetime}

\usepackage{geometry}
\newtheorem{theorem}[equation]{Theorem}
\newtheorem{lemma}[equation]{Lemma}

\newtheorem{corollary}[equation]{Corollary}

\theoremstyle{definition}

\numberwithin{equation}{section}

\colorlet{tn/color/defi}{red!50!black}

\newcommand{\chipcf}{\chi_{\rm pcf}}
\newcommand{\chio}{\chi_{\rm o}}

\begin{document}

  \title{Proper conflict-free 7-coloring of planar graphs}

  \author{A. Jiménez, C. N. Lintzmayer, and M. Sambinelli}

  \date{\today}

  \address[A. Jiménez]{Instituto de Ingeniería Matemática -- CIMFAV, Facultad de Ingeniería, Universidad de valparaíso, Valparaíso, Chile}
  \email{andrea.jimenez@uv.cl}

  \address[C. N. Lintzmayer]{Centro de Matemática, Computação e Cognição -- Universidade Federal do ABC, Santo André, Brazil}
  \email{carla.negri@ufabc.edu.br}

  \address[M. Sambinelli]{Centro de Matemática, Computação e Cognição -- Universidade Federal do ABC, Santo André, Brazil}
  \email{m.sambinelli@ufabc.edu.br}

  \begin{abstract}
    A proper conflict-free coloring is a proper vertex coloring in which every
    nonisolated vertex has a color occurring uniquely in its open
    neighborhood.
    We prove that every graph with neither a $K_5$-minor nor a $Q_6$-minor
    admits such a coloring with at most seven colors, where $Q_6=K_3\vee\overline{K_3}$.
    In particular, this improves the previous general upper bound of eight for
    planar graphs.
    The proof combines a previously developed iterated distance-three
    selector construction with a general anchor-contraction lifting principle.
    The first supplies independently colored witnesses in closed neighborhoods,
    while the second combines those witnesses with a proper coloring of a
    suitable minor.
    We also develop the parity analogue of the first mechanism and show that,
    whenever the $K_{k+1}$ case of Hadwiger's conjecture holds, every
    $K_{k+1}$-minor-free graph can be proper vertex colored with $2k-1$ colors such that every nonisolated vertex has a color occurring an odd number of times in its open neighborhood.
  \end{abstract}

  \maketitle

\section{Introduction}\label{sec:introduction}

Throughout this paper, all graphs are finite and simple.
For a vertex $v$ of a graph $G$, let $N_G(v)$ and $N_G[v]$ denote its open
and closed neighborhoods, respectively; for $D\subseteq V(G)$, put
$N_G[D] = \bigcup_{v\in D} N_G[v]$.
Also, let $[k] = \{1,\ldots,k\}$.
For vertex-disjoint graphs $G$ and $H$, their join $G\vee H$ is obtained
from their disjoint union followed by adding every edge between $V(G)$ and $V(H)$.

A coloring $c$ of $G$ is \emph{proper conflict-free} if it is proper and, for
every nonisolated vertex $v$, some color occurs exactly once in $N_G(v)$.
The least number of colors in such a coloring is the \emph{proper
conflict-free chromatic number} of $G$, denoted by $\chipcf(G)$.
This notion, introduced by Fabrici, Lu\v{z}ar, Rindo\v{s}ov\'{a}, and
Sot\'{a}k~\cite{FaLuRiSo23}, strengthens proper coloring while retaining the
local identification condition from conflict-free coloring.
The latter condition originated in frequency-assignment problems and has
since been studied in geometric and graph-theoretic
settings~\cite{EvLoRoSm03}.

Replacing uniqueness by parity gives the related notion of an \emph{odd
coloring}: a proper coloring in which, for every nonisolated vertex
$v$, some color occurs an odd number of times in $N_G(v)$.
Odd colorings were introduced by Petru\v{s}evski and
\v{S}krekovski~\cite{PeSk22}.
The minimum number of colors in a proper odd coloring is denoted by
$\chio(G)$.
Since a color that occurs exactly once also occurs an odd number of times, we
have
\begin{equation*}
  \chi(G)\le \chio(G)\le \chipcf(G)
\end{equation*}
for every graph $G$.

Proper conflict-free colorings of planar graphs were first studied by
Fabrici~\textit{et al.}~\cite{FaLuRiSo23}, who proved that every planar graph 
admits a proper conflict-free coloring with at most eight colors.
They also constructed a planar graph with no proper conflict-free
$5$-coloring and conjectured that six colors always suffice.
Subsequent work established stronger bounds for several sparse subclasses of
planar graphs, described in terms of girth or maximum average
degree~\cite{CaPeSk23}.
In particular, Cho~\textit{et al.}~\cite{ChChKwPa25} proved that every planar graph
of girth at least five is proper conflict-free $7$-colorable.
For general planar graphs, however, the gap between the lower bound~$6$ and
the upper bound~$8$ remained open.

A graph~$H$ is a \emph{minor} of a graph~$G$ if~$H$ can be obtained from~$G$
by deleting edges and/or vertices, and by contracting edges.
Very recently, Sharma, Paul, and Pandey~\cite{ShPaPa26} obtained an
$O(t\log\log t)$ upper bound for the proper conflict-free chromatic number of
$K_t$-minor-free graphs, improving the previously known quadratic dependence
on~$t$.
Their result applies for all~$t$, but does not subsume the explicit bounds
obtained here in the known range of Hadwiger's conjecture.
Our main result improves the general planar bound and holds for a larger
minor-closed class.

\begin{theorem}\label{thm:minor-main}
  Every graph~$G$ with neither a $K_5$-minor nor a $Q_6$-minor satisfies
  \begin{equation*}
    \chipcf(G)\le 7.
  \end{equation*}
\end{theorem}

\begin{corollary}\label{thm:main}
  Every planar graph $G$ satisfies
  \begin{equation*}
    \chipcf(G)\le7.
  \end{equation*}
\end{corollary}

In view of $\chio(G)\le\chipcf(G)$, Corollary~\ref{thm:main} also shows
that every planar graph has odd chromatic number at most~$7$, improving
the general upper bound of~$8$ due to Petr and Portier~\cite{PePo23}.

Our proof refines and combines two constructions from
Abel~\textit{et al.}~\cite{AbAlDeFeGoHeKeSc18}.
Their iterated elimination of distance-three sets produces the independent
closed-neighborhood selector used here, while their open-neighborhood
argument uses the same anchor-contraction device for obtaining unique
witnesses from a proper coloring of a minor.
We recast the first construction in selector language and formulate the
second as a general lifting principle for arbitrary admissible
multiplicities~$\Lambda$.
Combining the resulting selector with the Four Color Theorem yields the
bound $3+4=7$.

More precisely, Theorem~3.5 of Abel~\textit{et al.}~\cite{AbAlDeFeGoHeKeSc18},
in the equivalent independent-selector form recorded below, shows that
excluding $K_{k+2}$ and the near-complete graph $Q_{k+3}$ as minors produces
a $k$-color selector with a unique witness in every closed neighborhood.
Excluding~$K_5$ and~$Q_6$ allows us to take $k=3$, while the anchor minor
remains $K_5$-minor-free and hence $4$-colorable.
For the parity setting, we establish an analogous selector theorem requiring
only the exclusion of $K_{k+2}$.
Combined with Hadwiger's conjecture, it gives the bound~$2k-1$ for
$K_{k+1}$-minor-free graphs; the known cases yield new consequences for
graphs excluding~$K_5$ and~$K_6$ as minors.

Section~\ref{sec:selectors} develops the two types of closed-neighborhood
selector.
Section~\ref{sec:anchor-lifting} proves the anchor-lifting principle and
applies it first to planar proper conflict-free coloring and then to odd
coloring under complete-minor exclusions.

\section{Closed-neighborhood selectors}\label{sec:selectors}

Let $\Lambda$ be a set of positive integers with $1\in\Lambda$.
A \emph{$\Lambda$-selector of width $p$} in a graph $G$ is a pair
$(S,\varphi)$ consisting of an independent set $S\subseteq V(G)$ and a map
$\varphi:S\to[p]$ such that, for every $x\in V(G)$, there is a color
$i\in[p]$ for which
\begin{equation}\label{eq:selector}
  \bigl|N_G[x]\cap\varphi^{-1}(i)\bigr|\in\Lambda.
\end{equation}
If~\eqref{eq:selector} is required only for
$x\in V(G)\setminus S$, then $(S,\varphi)$ is called an
\emph{external $\Lambda$-selector}.
Under our standing assumptions the two notions coincide.
Indeed, if $s\in S$, then independence of~$S$ gives
$N_G[s]\cap S=\{s\}$, so the color $\varphi(s)$ occurs exactly once in
$N_G[s]$, and $1\in\Lambda$.
We use
\begin{equation*}
  \Lambda_{\mathrm{cf}}=\{1\}
  \qquad\text{and}\qquad
  \Lambda_{\mathrm{odd}}=\{1,3,5,\ldots\}.
\end{equation*}
Note that the selector property is componentwise: selectors of width~$p$ 
in the components can be combined into a selector of width~$p$ for 
the whole graph.

The next lemma isolates the geometric reduction underlying the iterated
distance-three construction of Abel~\textit{et al.}~\cite{AbAlDeFeGoHeKeSc18}.

\begin{lemma}
\label{lem:distance-three-reduction}
  Every nonempty connected graph $G$ contains a set $D\subseteq V(G)$ such
  that, with
  \begin{equation*}
    C=G[N_G[D]]    \qquad\text{and}\qquad   R=G-N_G[D],
  \end{equation*}
  the following properties hold:
  \begin{enumerate}
    \item distinct vertices of $D$ are at distance at least $3$;
    \item the graph $C$ is connected;
    \item $|N_G[x]\cap D|=1$ for every $x\in N_G[D]$;
    \item every vertex of $R$ has a neighbor in $C$ (thus $N_G[D]$ is a
              connected dominating set of $G$); and
    \item if a graph $J$ is a minor of $R$, then $K_1\vee J$ is a minor of
          $G$.
  \end{enumerate}
\end{lemma}
\begin{proof}
  Start $D$ with an arbitrary vertex of $G$ and, while there is a vertex at
  distance at least $3$ from the current set $D$, add a vertex at distance
  precisely $3$ from $D$.
  Such a vertex can be found on a shortest path from $D$ to any vertex at
  distance at least~$3$.
  Thus, distinct vertices of the resulting set $D$ are at distance at least~$3$,
  and every vertex of~$G$ is at distance at most~$2$ from~$D$.

  The graph $C$ is connected.
  Indeed, whenever a new vertex~$d$ is added, there is a path
  $(d,a,b,d')$ of length~$3$ to a previously chosen vertex $d'\in D$.
  The edge $ab$ joins $G[N_G[d]]$ to the subgraph induced by the closed
  neighborhoods of the vertices chosen earlier.

  Every $x\in N_G[D]$ has a vertex of $D$ in its closed neighborhood.
  It cannot have two, since two such vertices would be at distance at most
  $2$.
  This proves the third property.
  Every vertex of $R$ is at distance exactly $2$ from $D$, so it has a
  neighbor in $C$.

  Finally, suppose that $R$ contains $J$ as a minor.
  Contract the connected graph $C$ to a single vertex.
  Since every vertex of $R$ has a neighbor in $C$, the contracted vertex is
  adjacent to every branch set in a minor model of $J$ in $R$.
  These branch sets together with the contracted vertex form a
  $K_1\vee J$ minor of $G$.
\end{proof}

The next lemma isolates the common inductive lifting step.

\begin{lemma}\label{lem:selector-extension}
  Let $G$ be a graph, and let $D\subseteq V(G)$ be a set whose distinct
  vertices are pairwise at distance at least $3$, and suppose that
  $|N_G[x]\cap D|=1$
  for every $x\in N_G[D]$.
  Put $R=G-N_G[D]$.
  If $R$ has a $\Lambda$-selector of width $p$, then $G$ has a
  $\Lambda$-selector of width $p+1$.
\end{lemma}
\begin{proof}
  Let $(S_R,\varphi_R)$ be a $\Lambda$-selector of width $p$ in $R$.
  Put $S=D\cup S_R$, retain $\varphi_R$ on $S_R$, and assign color $p+1$ to
  every vertex of $D$.
  The set $S$ is independent: the set $D$ is a distance-three packing,
  $S_R$ is independent, and no edge joins $D$ to $R$.

  If $x\in N_G[D]$, then the new color occurs exactly once in $N_G[x]$,
  and $1\in\Lambda$.
  If $x\in V(R)$, then
  \begin{equation*}
    N_G[x]\cap S_R=N_R[x]\cap S_R,
  \end{equation*}
  so the selector in $R$ supplies the required color.
\end{proof}

For $t\ge4$, let $Q_t$ be obtained from $K_t$ by deleting the three edges
among a fixed triple of vertices. Thus
\begin{equation}\label{eq:q-cone}
  Q_t=K_{t-3}\vee\overline{K_3},
  \qquad
  Q_4=K_{1,3},
  \qquad
  Q_{t+1}=K_1\vee Q_t.
\end{equation}

The next lemma is an equivalent independent-support reformulation of
Theorem~3.5 of Abel~\textit{et al.}~\cite{AbAlDeFeGoHeKeSc18}.
Their graph $K_{k+3}^{-3}$ is precisely $Q_{k+3}$, and the support produced
by their iterated distance-three construction is independent.
We include the proof to make the selector formulation explicit and because
the same inductive structure will be used for the parity selector.

\begin{lemma}\label{lem:stable-dominating}
  Let $k\ge1$.
  If a graph $G$ has neither a $K_{k+2}$-minor nor a $Q_{k+3}$-minor, then
  $G$ has a $\Lambda_{\mathrm{cf}}$-selector of width $k$.
\end{lemma}
\begin{proof}
  We proceed by induction on $k$.
  Since selectors combine componentwise, we may assume that $G$ is connected.

  Suppose first that $k=1$.
  Since $G$ has no $K_3$-minor, it is a tree.
  Since it has no $Q_4=K_{1,3}$ minor, it has maximum degree at most $2$ and
  hence is a path, say $(v_1,\ldots,v_n)$.
  If $n\equiv0,2\pmod3$, let $S$ consist of the vertices $v_j$ with
  $j\equiv2\pmod3$; if $n\equiv1\pmod3$, use instead the vertices $v_j$ with
  $j\equiv1\pmod3$.
  Every closed neighborhood contains exactly one vertex of $S$, so assigning
  color $1$ to $S$ proves the base case.

  Let $k\ge2$, and take $D,C,R$ from
  Lemma~\ref{lem:distance-three-reduction}.
  If $R$ contained a $K_{k+1}$-minor, then $G$ would contain
  $K_1\vee K_{k+1}=K_{k+2}$
  as a minor.
  Similarly, a $Q_{k+2}$-minor in $R$ would give a
  $K_1\vee Q_{k+2}=Q_{k+3}$ minor in $G$.
  Thus, the induction hypothesis gives a
  $\Lambda_{\mathrm{cf}}$-selector of width $k-1$ in $R$.
  Lemma~\ref{lem:selector-extension} lifts it to a selector of width $k$ in
  $G$.
\end{proof}

The parity base case is a known consequence of the existence of
odd-dominating sets; see Proposition~2.1 and the proof of Theorem~3.8 of
Caro, Petru\v{s}evski, and \v{S}krekovski~\cite{CaPeSk23b}.
We record it in the form needed here.

\begin{lemma}[\cite{CaPeSk23b}]
  \label{lem:forest-independent-odd-set}
  Every forest $F$ has an independent set $S\subseteq V(F)$ such that
  \begin{equation}\label{eq:forest-closed-odd}
    |N_F[v]\cap S| \equiv 1 \pmod 2
  \end{equation}
  for every $v\in V(F)$.
\end{lemma}
\begin{proof}
  By the odd-domination theorem of Sutner, $F$ has a set
  $S\subseteq V(F)$ such that $|N_F[v]\cap S|$ is odd for every
  $v\in V(F)$; see~\cite{CaPeSk23b}.
  For every $v\in S$,
  \begin{equation*}
    d_{F[S]}(v)=|N_F[v]\cap S|-1
  \end{equation*}
  is even.
  Hence every vertex of the forest $F[S]$ has even degree.
  If $F[S]$ contained an edge, one of its nontrivial components would have a
  vertex of degree one, a contradiction.
  Thus $S$ is independent.
\end{proof}

\begin{lemma}
  \label{lem:stableset-odd-coloring}
  Let $k\ge1$.
  If a graph $G$ has no $K_{k+2}$-minor, then $G$ has a
  $\Lambda_{\mathrm{odd}}$-selector of width $k$.
\end{lemma}
\begin{proof}
  We proceed by induction on $k$.
  The assertion is trivial if $G$ is empty, and selectors combine
  componentwise, so we may assume that $G$ is connected.
  If $k=1$, then $G$ is a tree, and
  Lemma~\ref{lem:forest-independent-odd-set} gives the desired selector.

  Let $k\ge2$, and take $D,C,R$ from
  Lemma~\ref{lem:distance-three-reduction}.
  A $K_{k+1}$-minor in $R$ would extend to a
  $K_1\vee K_{k+1}=K_{k+2}$ minor in $G$.
  Hence~$R$ has no $K_{k+1}$-minor.
  The induction hypothesis gives a $\Lambda_{\mathrm{odd}}$-selector of
  width~$k-1$ in~$R$, and Lemma~\ref{lem:selector-extension} lifts it to~$G$.
\end{proof}

\section{Anchor lifting and applications}\label{sec:anchor-lifting}

We now formulate a general lifting principle from closed-neighborhood
selectors to proper colorings with open-neighborhood witnesses.
The underlying edge-contraction device appears in the proofs of
Abel~\textit{et al.}~\cite{AbAlDeFeGoHeKeSc18} and
Bhyravarapu, Kalyanasundaram, and Mathew~\cite{BhKaMa22}.
The formulation below applies it to an arbitrary set of admissible
multiplicities~$\Lambda$ and separates the selector width from the number
of colors used on the resulting minor.

\begin{lemma}\label{lem:anchor-lifting}
  Let $(S,\varphi)$ be an external $\Lambda$-selector of width $p$ in a
  graph $G$, and put $U=V(G)\setminus S$.
  For every nonisolated $s\in S$, choose an anchor $a_s\in N_G(s)$.
  Let $H$ be the simple graph obtained by contracting every edge $sa_s$ and
  deleting the isolated vertices of $S$.
  If $H$ is $q$-colorable, then $G$ has a proper coloring with at most
  $p+q$ colors such that every nonisolated vertex has a color whose
  multiplicity in its open neighborhood belongs to $\Lambda$.
\end{lemma}
\begin{proof}
  Since $S$ is independent, every anchor belongs to $U$, and the chosen
  edges form a disjoint union of stars centered at vertices of $U$.
  Thus, distinct vertices of $U$ remain distinct under the contractions, and
  $V(H)$ is naturally identified with $U$.
  Moreover,
  \begin{equation}\label{eq:anchor-properties}
    G[U]\subseteq H,
    \qquad
    {a_s}x\in E(H)
  \end{equation}
  for every nonisolated $s\in S$ and every
  $x\in N_G(s)\setminus\{a_s\}$.

  Let $\psi:U\to\{p+1,\ldots,p+q\}$ be a proper coloring of $H$, and define
  $c$ by retaining $\varphi$ on $S$ and $\psi$ on $U$.
  The coloring $c$ is proper: the set $S$ is independent, the two palettes
  are disjoint, and $G[U]\subseteq H$.

  Let $u\in U$.
  Since $u\notin S$, for each $i\in[p]$ we have
  \begin{equation*}
    N_G[u]\cap\varphi^{-1}(i)
    =
    N_G(u)\cap\varphi^{-1}(i).
  \end{equation*}
  Hence the color supplied by the selector has the same multiplicity,
  belonging to $\Lambda$, in $N_G(u)$.

  Finally, let $s\in S$ be nonisolated.
  By~\eqref{eq:anchor-properties}, every other neighbor
  $x\in N_G(s)\setminus\{a_s\}$ is adjacent to~$a_s$ in~$H$.
  Therefore $\psi(x)\ne\psi(a_s)$, so $\psi(a_s)$ occurs exactly once in
  $N_G(s)$.
  This is a valid witness because $1\in\Lambda$.
\end{proof}

\begin{corollary}\label{cor:minor-closed-transfer}
  Let $\mathcal C$ be a minor-closed class all of whose graphs are
  $q$-colorable, and let $G\in\mathcal C$.
  \begin{enumerate}
    \item If $G$ has an external $\Lambda_{\mathrm{cf}}$-selector of width~$p$, then $\chipcf(G)\le p+q$.
    \item If $G$ has an external $\Lambda_{\mathrm{odd}}$-selector of width~$p$, then $\chio(G)\le p+q$.
  \end{enumerate}
\end{corollary}

\begin{proof}
  Form the anchor graph $H$ from Lemma~\ref{lem:anchor-lifting}.
  It is a minor of $G$, and hence belongs to $\mathcal C$ and is
  $q$-colorable.
  Apply the lemma with $\Lambda=\Lambda_{\mathrm{cf}}$ or
  $\Lambda=\Lambda_{\mathrm{odd}}$, respectively.
\end{proof}

We shall use the following known cases of Hadwiger's conjecture.

\begin{theorem}[Hadwiger's conjecture for $t\le6$]
  \label{thm:known-hadwiger}
  If $2\le t\le6$, then every $K_t$-minor-free graph is
  $(t-1)$-colorable.
\end{theorem}

The cases $t\le4$ are elementary.
The case $t=5$ follows from Wagner's structure theorem and the Four Color
Theorem~\cites{WagnerTheorem,ApHa77a,ApHa77b}, and the case $t=6$ is due to
Robertson, Seymour, and Thomas~\cite{RoSeTh93}.

\subsection{Proper conflict-free coloring and complete minors}

\begin{corollary}\label{cor:known-hadwiger-pcf}
  Let $3\le t\le6$.
  \begin{enumerate}
    \item Every $K_t$-minor-free graph $G$ satisfies
    $\chipcf(G) \le 2t-2$.
    \item If $G$ has neither a $K_t$-minor nor a $Q_{t+1}$-minor, then
    $\chipcf(G) \le 2t-3$.
  \end{enumerate}
\end{corollary}
\begin{proof}
  Let $\mathcal C$ be the minor-closed class of $K_t$-minor-free graphs.
  By Theorem~\ref{thm:known-hadwiger}, every graph in $\mathcal C$ is
  $(t-1)$-colorable.

  For the first assertion, observe that a $K_t$-minor-free graph has neither
  a $K_{t+1}$-minor nor a $Q_{t+2}$-minor, since $Q_{t+2}$ contains $K_t$ as
  a subgraph.
  Lemma~\ref{lem:stable-dominating}, with parameter $t-1$, gives a
  $\Lambda_{\mathrm{cf}}$-selector of width $t-1$.
  Corollary~\ref{cor:minor-closed-transfer} gives the bound
  $(t-1)+(t-1)=2t-2$.

  Under the additional exclusion of $Q_{t+1}$,
  Lemma~\ref{lem:stable-dominating}, with parameter $t-2$, instead gives a
  selector of width $t-2$.
  The resulting bound is $(t-2)+(t-1)=2t-3$.
\end{proof}

For $t=3$ and $t=4$, the bounds in
Corollary~\ref{cor:known-hadwiger-pcf} are not new and its first bound is
not best possible.
Every forest is proper conflict-free $3$-colorable~\cite{CaPeSk23}, 
while the list-coloring result of Liu~\cite{Liu24} implies that 
every $K_4$-minor-free graph is proper conflict-free $5$-colorable, 
without the additional exclusion of~$Q_5$.
Consequently, the cases of the corollary not covered by these earlier
results are $t=5$ and $t=6$.

\begin{proof}[Proof of Theorem~\ref{thm:minor-main}]
  Apply part~(2) of Corollary~\ref{cor:known-hadwiger-pcf} with $t=5$.
\end{proof}

\begin{proof}[Proof of Corollary~\ref{thm:main}]
  Let $G$ be planar.
  Since planarity is closed under taking minors, $G$ has neither a
  $K_5$-minor nor a $K_{3,3}$-minor.
  The graph $K_{3,3}$ is a spanning subgraph of $Q_6$, so $G$ also has no
  $Q_6$-minor.
  Theorem~\ref{thm:minor-main} now applies.
\end{proof}

\subsection{Odd coloring and Hadwiger's conjecture}

\begin{lemma}[Hadwiger transfer]\label{lem:hadwiger-odd-transfer}
  Let $k\ge1$, and suppose that every $K_{k+1}$-minor-free graph is
  $k$-colorable.
  Then every $K_{k+1}$-minor-free graph $G$ satisfies
  $\chio(G) \le 2k-1$.
\end{lemma}
\begin{proof}
  The case $k=1$ is immediate, since a $K_2$-minor-free graph is edgeless.
  Let $k\ge2$.
  Lemma~\ref{lem:stableset-odd-coloring}, applied with parameter $k-1$,
  gives a $\Lambda_{\mathrm{odd}}$-selector of width $k-1$.
  The $K_{k+1}$-minor-free graphs form a minor-closed class and are
  $k$-colorable by hypothesis.
  Corollary~\ref{cor:minor-closed-transfer} yields
  $\chio(G) \le (k-1)+k = 2k-1$.
\end{proof}

\begin{corollary}\label{cor:known-hadwiger-odd}
  Let $2\le t\le6$.
  Every $K_t$-minor-free graph $G$ satisfies
  \begin{equation*}
    \chio(G)\le2t-3.
  \end{equation*}
\end{corollary}
\begin{proof}
  Combine Theorem~\ref{thm:known-hadwiger} with
  Lemma~\ref{lem:hadwiger-odd-transfer}, taking $k=t-1$.
\end{proof}

The cases $t\le4$ of Corollary~\ref{cor:known-hadwiger-odd} were also previously
known.
The cases $t\le3$ follow from the corresponding results for forests, and
Cranston, Lafferty, and Song~\cite{CrLaSo23} proved that every
$K_4$-minor-free graph is odd $5$-colorable.
Thus the new complete-minor consequences of
Lemma~\ref{lem:hadwiger-odd-transfer} are the cases $t=5$ and $t=6$.

\begin{corollary}\label{cor:planar-odd-seven-from-hadwiger}
  Every planar graph $G$ satisfies
  \begin{equation*}
    \chio(G)\le7.
  \end{equation*}
\end{corollary}

\begin{proof}
  Every planar graph is $K_5$-minor-free.
  Apply Corollary~\ref{cor:known-hadwiger-odd} with $t=5$.
  The same conclusion also follows from Corollary~\ref{thm:main} and
  $\chio(G)\le\chipcf(G)$.
\end{proof}

\section*{Acknowledgements}

This research has been partially supported by Coordena\c{c}{\~a}o de
Aperfei\c{c}oamento de Pessoal de N{\'i}vel Superior (CAPES), Finance Code 001.
C. N. Lintzmayer was partially supported by the Conselho Nacional de
Desenvolvimento Cient{\'i}fico e Tecnol{\'o}gico (CNPq), grants 404315/2023-2
and 301143/2025-0, and by the Fundação de Amparo à Pesquisa do Estado de S\~{a}o
Paulo (FAPESP), grant 2024/16092-6.
M. Sambinelli was also partially supported by FAPESP, grant 2025/00323-1.

\section*{Statement on the use of artificial intelligence}

Generative artificial-intelligence tools were used during the development
of this work to assist with mathematical exploration, computational
experimentation, the organization and refinement of arguments, and language
editing.
Every statement and proof included in the paper was checked by the authors,
who take full responsibility for the content.

\bibliographystyle{amsplain}

\end{document}